\documentclass{amsart}

\usepackage[a4paper,hmargin=2.8cm,vmargin=4cm]{geometry}
\usepackage{amsfonts,amssymb,amscd,amstext}
\usepackage{mathtools}
\mathtoolsset{showonlyrefs=true} 
\usepackage[utf8]{inputenc}
\usepackage[colorlinks=true,linkcolor=blue,citecolor=red]{hyperref}
\usepackage{braket,relsize,nccmath}

\makeatletter
\@namedef{subjclassname@2020}{
  \textup{2020} Mathematics Subject Classification}
\makeatother

\renewcommand{\geq}{\geqslant}
\newcommand{\ptl}{\partial}
\newcommand{\rr}{{\mathbb{R}}}
\newcommand{\rrn}{\mathbb{R}^{n+1}}
\newcommand{\la}{\lambda}
\newcommand{\sph}{\mathbb{S}}
\newcommand{\escpr}[1]{\big<#1\big>}
\newcommand{\Sg}{\Sigma} 
\newcommand{\sg}{\sigma}
\newcommand{\Om}{\Omega}
\newcommand{\eps}{\varepsilon}
\newcommand{\var}{\varphi}
\newcommand{\de}{\mathcal{D}}
\newcommand{\cc}{\mathcal{C}}
\newcommand{\ovn}{\overline{N}}
\newcommand{\ww}{\mathcal{W}}

\DeclareMathOperator{\divv}{div}

\newtheorem{theorem}{Theorem}[section]
\newtheorem{proposition}[theorem]{Proposition}
\newtheorem{lemma}[theorem]{Lemma}
\newtheorem{corollary}[theorem]{Corollary}

\theoremstyle{definition}

\newtheorem{remark}[theorem]{Remark}
\newtheorem{remarks}[theorem]{Remarks}
\newtheorem{example}[theorem]{Example}
\newtheorem{examples}[theorem]{Examples}

\theoremstyle{remark}

\numberwithin{equation}{section}

\begin{document}

\title[Anisotropic stability in convex solid cones]{Compact anisotropic stable hypersurfaces with free boundary in convex solid cones}

\author[C\'esar Rosales]{C\'esar Rosales}
\address{Departamento de Geometr\'{\i}a y Topolog\'{\i}a and Excellence Research Unit ``Modeling Nature'' (MNat) Universidad de Granada, E-18071,
Spain.} 
\email{crosales@ugr.es}

\date{February 13, 2023}

\thanks{The author was supported by the research grant PID2020-118180GB-I00 funded by MCIN/AEI/10.13039/501100011033 and the Junta de Andaluc\'ia grant PY20-00164.} 



\begin{abstract}
We consider a convex solid cone $\cc\subset\rrn$ with vertex at the origin and boundary $\ptl\cc$ smooth away from $0$. Our main result shows that a compact two-sided hypersurface $\Sg$ immersed in $\cc$ with free boundary in $\ptl\cc\setminus\{0\}$ and minimizing, up to second order, an anisotropic area functional under a volume constraint is contained in a Wulff-shape.
\end{abstract}

\maketitle

\thispagestyle{empty}

\section{Introduction}
\label{sec:intro}
\setcounter{equation}{0}

In this paper we study a variational problem for Euclidean hypersurfaces associated to an energy functional of \emph{anisotropic character}. This means that the energy is computed by integrating an elliptic parametric function (or surface tension) which depends on the normal direction along the hypersurface. As it is explained in the introduction of Taylor~\cite{taylor} these functionals provide a mathematical model to study solid crystals. In the particular case of a constant surface tension we obtain the isotropic case, where the energy is proportional to the Euclidean area.

The surface tension that we consider in this work is the asymmetric norm given by the support function $h_K$ of a smooth strictly convex body $K\subset\rrn$ containing the origin in its interior. The corresponding \emph{anisotropic area} of a two-sided hypersurface $\Sg$ with unit normal $N$ has the expression  $A_K(\Sg):=\int_\Sg h_K(N)\,d\Sg$, where $d\Sg$ is the Euclidean area element. By using the metric projection onto $K$ it is possible to define also anisotropic counterparts to some of the classical notions in the extrinsic geometry of hypersurfaces, see Reilly~\cite{reilly-anisotropic} and Section~\ref{sec:prelimi} for a precise description. When $K$ is the round unit ball about the origin we have $h_K(N)=1$, so that we recover the isotropic situation. We point out that the anisotropic geometry of hypersurfaces is typically introduced by means of a function $F:\sph^n\to\rr^+$ instead of a convex body $K$, see Remark~\ref{re:other} for more details.

It is well known that the round spheres uniquely solve the isoperimetric problem in $\rrn$. In our anisotropic setting, the strictly convex hypersurface $\ptl K$ uniquely minimizes, up to translations and dilations centered at the origin, the anisotropic area computed with respect to the outer unit normal among hypersurfaces enclosing the same Euclidean volume. There are different proofs of this statement relying on analytical and geometric techniques, see Taylor~\cite{taylor-unique,taylor}, Fonseca and Müller~\cite{fonseca-muller}, Brothers and Morgan~\cite{brothers-morgan}, Milman and Rotem~\cite{milman-rotem}, and Cabr\'e, Ros-Oton and Serra~\cite{cabre-rosoton-serra}. The minimizer $\ptl K$ is usually referred to as the \emph{Wulff shape} in honor of the crystallographer Georg Wulff, who first constructed the optimal crystal for a specified integrand~\cite{wulff}. 

As in the isotropic case, it is interesting to analyze the critical points (stationary hypersurfaces) and the second order minima (stable hypersurfaces) of the anisotropic area under a volume constraint. From the first variational formulas, a hypersurface $\Sg$ is \emph{anisotropic stationary} if and only if its \emph{anisotropic mean curvature} $H_K$ defined in \eqref{eq:mc} is constant, see for instance Palmer~\cite{palmer} or Clarenz~\cite{clarenz-2002}. This motivated the generalization to the anisotropic context of some classical results for constant mean curvature hypersurfaces in $\rrn$. In this direction, an Alexandrov type theorem (resp. a Hopf type theorem) characterizing the Wulff shape among compact embedded hypersurfaces (resp. immersed spheres) with $H_K$ constant was obtained by Morgan~\cite{morgan-alexandrov} for $n=1$, and by He, Li, Ma and Ge~\cite{alexandrov-anisotropic} in higher dimension (resp. by He and Li~\cite{he-li-hopf}, and Koiso and Palmer~\cite{koiso-palmer-hopf}). On the other hand, Palmer~\cite{palmer} and Winklmann~\cite{winklmann} employed the second variation of the anisotropic area to show that, up to translation and homothety, the Wulff shape $\ptl K$ is the unique compact, two-sided,  \emph{anisotropic stable hypersurface} immersed in $\rrn$. This is an anisotropic extension of a celebrated theorem of Barbosa and do Carmo~\cite{bdc} also discussed by Wente~\cite{wente}.

When we consider a smooth proper domain $\Om$ instead of the whole space $\rrn$, we are naturally led to study the \emph{anisotropic partitioning problem}. Here, we seek minimizers of the functional $A_K$ among hypersurfaces inside $\overline{\Om}$, possibly with non-empty boundary contained in $\ptl\Om$, and separating a fixed Euclidean volume. The critical points for this problem are called anisotropic stationary hypersurfaces \emph{with free boundary} in $\ptl\Om$. From the calculus of $A_K'(0)$, and reasoning as Koiso and Palmer~\cite{koiso-palmer2}, it follows that a hypersurface $\Sg$ is anisotropic stationary with free boundary in $\ptl\Om$ if and only if $H_K$ is constant on $\Sg$ and the anisotropic normal $N_K$ introduced in \eqref{eq:normal} is tangent to $\ptl\Om$ along $\ptl\Sg$. The more general case of \emph{anisotropic capillary hypersurfaces} arises when both, $H_K$ and the angle between $N_K$ and the inner unit normal to $\ptl\Om$, are constant. These are critical points of an energy functional which involves not only $A_K(\Sg)$ but also the wetting area of the set bounded by $\ptl\Sg$ in $\ptl\Om$. We remark that isotropic capillary hypersurfaces in convex domains of $\rrn$ have been extensively investigated, see for instance Ros and Souam~\cite{ros-souam}, Wang and Xia~\cite{wang-xia}, and the references therein.

An anisotropic stationary hypersurface with free boundary in $\ptl\Om$ is \emph{stable} if $A''_K(0)$ is nonnegative under variations preserving the volume separated by $\Sg$ and the boundary of $\Om$. The computation of $A_K''(0)$ and the subsequent analysis of anisotropic stable hypersurfaces have been treated in several previous works. In \cite{koiso-palmer2,koiso-palmer4,koiso-palmer3}, Koiso and Palmer considered stable capillary surfaces when $\Om$ is a slab of $\rr^3$ and the Wulff shape is rotationally symmetric about a line orthogonal to $\ptl\Om$. In \cite{koiso-equilibrium,koiso-2023}, Koiso classified compact anisotropic stable capillary hypersurfaces disjoint from the edges in wedge-shaped domains of $\rrn$. Inside a smooth domain of revolution $\Om\subset\rr^3$, and for certain rotationally symmetric surface tensions, Barbosa and Silva~\cite{barbosa-silva} established that the totally geodesic disks orthogonal to the revolution axis are the unique compact stationary surfaces with free boundary in $\ptl\Om$ and meeting $\ptl\Om$ orthogonally. By assuming convexity of $\Om$ they also discussed the stability of these surfaces. More recently, Jia, Wang, Xia and Zhang~\cite{jwxz} have employed a Heintze-Karcher inequality and a Minkowski formula to prove that any compact, embedded, anisotropic capillary hypersurface inside an open half-space of $\rrn$ is part of a Wulff shape. In \cite{guo-xia}, besides the computation of a very general second variation formula, Guo and Xia have generalized the argument of Winklmann~\cite{winklmann} after Barbosa and do Carmo~\cite{bdc} to show that a compact, two-sided, anisotropic stable capillary hypersurface immersed in a Euclidean half-space of $\rrn$ is a truncated Wulff shape.

In this paper we are interested in the stability question when the ambient domain is a solid cone $\cc\subset\rrn$ with vertex at the origin and boundary $\ptl\cc$ smooth away from $0$. In the isotropic case it is known that round balls about the vertex provide the unique solutions to the partitioning problem in a (possibly non-smooth) convex cone $\cc$, up to translations leaving $\cc$ invariant, see Lions and Pacella~\cite{lp}, and Figalli and Indrei~\cite{figalli-indrei}. When $\ptl\cc\setminus\{0\}$ is smooth, Ritor\'e and the author~\cite{cones} solved the problem by combining an existence result with the classification of compact stable hypersurfaces with free boundary in $\ptl\cc$. We must also mention the papers of Choe and Park~\cite{choe-cones}, and Pacella and Tralli~\cite{pacella-tralli}, where an Alexandrov type theorem for compact, embedded, stationary hypersurfaces in convex cones is proven. 

In the anisotropic case it is easy to check that the truncated Wulff shape $\ptl K\cap\overline{\cc}$ is an anisotropic stationary hypersurface with free boundary in $\ptl\cc$, see Example~\ref{ex:typical}. In \cite{weng}, Weng found some conditions ensuring that, up to translation and homothety, this is the unique compact and embedded hypersurface with constant anisotropic mean curvature in $\cc$. When the cone $\cc$ is convex, results of Milman and Rotem~\cite{milman-rotem}, and Cabr\'e, Ros-Oton and Serra~\cite{cabre-rosoton-serra} for anisotropic area functionals with homogeneous weights entail that $\ptl K\cap\overline{\cc}$ minimizes the anisotropic area (computed with respect to the outer unit normal) for fixed volume. We remark that the techniques in \cite{milman-rotem} and \cite{cabre-rosoton-serra} still hold under weaker regularity assumptions on the cone $\cc$ and the surface tension. By extending arguments of Figalli and Indrei~\cite{figalli-indrei}, the uniqueness of $\ptl K\cap\overline{\cc}$ as a solution to the partitioning problem in $\cc$ was recently discussed by Dipierro, Poggesi and Valdinoci \cite[Thm.~4.2]{dipierro-poggesi-valdinoci}. The fact that $\ptl K\cap\overline{\cc}$ is a minimizer implies, in particular, the weaker property that it is an anisotropic stable hypersurface with free boundary in $\ptl\cc$. Our main result, Theorem~\ref{th:main}, is the following uniqueness statement:
\begin{quotation}
\emph{Up to a translation and a dilation centered at the vertex, any compact, connected, two-sided, anisotropic stable hypersurface immersed in a convex solid cone $\cc$ with free boundary in $\ptl\cc\setminus\{0\}$ is part of the Wulff shape.}
\end{quotation}
We must observe that, in general, we cannot expect $\ptl K\cap\overline{\cc}$ to be the only anisotropic stable hypersurface, up to translation and homothety, see Example~\ref{ex:nonuniqueness}. This motivates us to find additional conditions on the cone in order to deduce stronger rigidity consequences. With this idea in mind, in Corollary~\ref{cor:stronger} we prove that a compact anisotropic stable hypersurface in a cone $\cc$ over a smooth strictly convex domain of the sphere $\sph^n$ must be a dilation of $\ptl K\cap\overline{\cc}$. On the other hand, for the case of a Euclidean half-space $\mathcal{H}$, recently treated by Guo and Xia~\cite{guo-xia}, we can conclude that $\ptl K\cap\overline{\mathcal{H}}$ is the only compact anisotropic stable hypersurface, up to translations along the boundary hyperplane.

Our proof of Theorem~\ref{th:main} is different from the ones in \cite{cones} and \cite{guo-xia}, which are mainly inspired in Barbosa and do Carmo~\cite[Thm.~(1.3)]{bdc}. Instead, we adapt to our setting the idea employed by Wente~\cite{wente} to characterize the round spheres as the only compact, two-sided, stable hypersurfaces immersed in $\rrn$. Wente applied the stability inequality with the special volume-preserving deformation obtained by parallel hypersurfaces dilated to keep the volume constant. After computing the second variation formulas for this particular variation, he observed that the area decreases strictly unless the hypersurface is a round sphere. By using dilations of \emph{anisotropic parallel hypersurfaces}, Palmer~\cite{palmer} proved the uniqueness of the Wulff shape as a compact, two-sided, anisotropic stable hypersurface immersed in $\rrn$. The variation of a hypersurface $\Sg$ by anisotropic parallels is defined as $\psi_t(p):=p+t\,N_K(p)$, where $p\in\Sg$, $t\in\rr$ and $N_K$ is the anisotropic normal introduced in \eqref{eq:normal}. With a similar argument, Koiso~\cite[Thm.~4]{koiso-equilibrium}, \cite[Thm.~1]{koiso-2023} was able to classify compact anisotropic stable hypersurfaces disjoint from the edges in a wedge-shape domain $\Om$ of $\rrn$. Moreover, after a suitable translation, the same variation allows to treat the capillary case. We remark that such a deformation is possible because $\ptl\Om$ consists of finitely many pieces of hyperplanes.

In a convex cone $\cc$ different from a half-space, Palmer's deformation does not work in general. Indeed, though $N_K$ is tangent to $\ptl\cc$ along $\ptl\Sg$, we cannot ensure that $\ptl\cc$ contains a small anisotropic normal segment $p+t\,N_K(p)$ centered at any point $p\in\ptl\Sg$. To solve this difficulty we replace the anisotropic parallels deformation of $\Sg$ with a variation $\Sg_t:=\psi_t(\Sg)$ associated to the one-parameter group $\{\psi_t\}_{t\in\rr}$ of a vector field $X$ on $\rrn$ such that $X(0)=0$, $X$ is tangent on $\ptl\cc\setminus\{0\}$ and $X_{|\Sg}=N_K$. After this, as $\ptl\cc$ is invariant under dilations centered at $0$, we can apply to $\Sg_t$ a suitable dilation in order to construct a deformation that preserves both, $\ptl\cc$ and the volume of $\Sg$. To finish the proof we compute the second derivative of the anisotropic area for this variation, and we discover that it is strictly negative unless $\Sg$ is contained in the Wulff shape (up to translation and homothety). 

We emphasize that, to prove Theorem~\ref{th:main}, we only need to calculate $A_K''(0)$ for deformations associated to a smooth vector field $X$ on $\rrn$ which is tangent on $\ptl\cc\setminus\{0\}$ and satisfies $X_{|\Sg}=N_K$. This is done in Proposition~\ref{prop:2ndarea} for arbitrary compact hypersurfaces with non-empty boundary in $\rrn$. As a difference with respect to the anisotropic parallel deformation, a boundary term involving the acceleration vector field $Z$ appears. For anisotropic stationary hypersurfaces in $\cc$ we see in Proposition~\ref{prop:2ndareavol} that this term is related to the extrinsic Euclidean geometry of $\ptl\cc$. As the formula for $A_K''(0)$ in Proposition~\ref{prop:2ndareavol} is valid for any smooth domain $\Om$, this shows that the stability condition is more restrictive when $\Om$ is convex. It is worth mentioning that, in most of the aforementioned works about anisotropic capillary hypersurfaces in a Euclidean domain $\Om$, this boundary term had no relevance since $\ptl\Om$ is totally geodesic.

Finally, we remark that the proof of Theorem~\ref{th:main} is still valid in non-smooth convex cones, provided the boundary $\ptl\Sg$ of an anisotropic stable hypersurface $\Sg$ is inside a smooth open part of $\ptl\cc$. 

The paper is organized into three sections besides this introduction. In Section~\ref{sec:prelimi} we review some basic facts about the anisotropic geometry of hypersurfaces. In Section~\ref{sec:varform} we compute the second derivative of the anisotropic area for certain deformations of a compact hypersurface with non-empty boundary. Finally, in Section~\ref{sec:main} we prove our classification result of anisotropic stable hypersurfaces in convex solid cones.

\section{Preliminaries}
\label{sec:prelimi}
\setcounter{equation}{0}

In this section we gather some definitions and results that will be needed throughout this work. Starting with a Euclidean convex body we will use its support function and metric projection to introduce the anisotropic area and notions of anisotropic extrinsic geometry for two-sided hypersurfaces. The relation between this point of view and previous approaches is shown in Remark~\ref{re:other}.  

By a \emph{convex body} (about the origin) we mean a compact convex set $K\subset\rrn$ containing $0$ in its interior. The associated \emph{support function} $h_K:\rrn\to\rr$ is given by
\[
h_K(w):=\max\{\escpr{u,w}\,;\,u\in K\}, 
\]
where $\escpr{\cdot\,,\cdot}$ is the standard scalar product in $\rrn$. This defines an \emph{asymmetric norm} in $\rrn$ (which is a norm when $K$ is centrally symmetric about $0$), see \cite[Sect.~1.7.1]{schneider}. For any $w\neq 0$, the \emph{supporting hyperplane} of $K$ with exterior normal $w$ is the set $\Pi_w:=\{p\in\rrn\,;\,\escpr{p,w}=h_K(w)\}$. The corresponding \emph{support set} is $\Pi_w\cap K$. This set is not empty because $K$ is compact. If $\Pi_w\cap K$ is a single point $p$, then $h_K$ is differentiable at $w$, and its gradient satisfies $(\nabla h_K)(w)=p$, see \cite[Cor.~1.7.3]{schneider}. 

Henceforth, we suppose that $K$ is a strictly convex body with smooth boundary $\ptl K$. Thus, for any $w\neq 0$, there is a point $\pi_K(w)\in\ptl K$ for which $\Pi_w\cap K=\{\pi_K(w)\}$. This provides a map $\pi_K:\rrn_*\to\ptl K$, which is called the \emph{$K$-projection}, and verifies these identities
\begin{equation}
\label{eq:gradhK}
\begin{split}
\escpr{\pi_K(w),w}&=h_K(w),\quad\text{for any }  w\neq 0,
\\
(\nabla h_K)(w)&=\pi_K(w), \quad\text{for any }  w\neq 0.
\end{split}
\end{equation}
On the other hand, if $\eta_K$ is the outer unit normal on $\ptl K$, then $(\eta_K\circ\pi_K)(w)=w$ for any $w\in\sph^n$. This equality and the fact that $\pi_K(\la\,w)=\pi_K(w)$ for any $w\neq 0$ and $\la>0$ entail that $\pi_K(w)=\eta_K^{-1}(w/|w|)$ for any $w\neq 0$ (here $|\cdot|$ stands for the Euclidean norm). As $\eta_K:\ptl K\to\sph^n$ is a diffeomorphism, the support function $h_K$ is smooth on $\rrn_*$. Observe that, when $K$ is the round unit ball about $0$, then $h_K(w)=|w|$ and $\pi_K(w)=w/|w|$ for any $w\neq 0$, whereas $\eta_K(p)=p$ for any $p\in\ptl K=\sph^n$.

Let $\var_0:\Sg\to\rrn$ be a smooth two-sided immersed hypersurface, possibly with smooth boundary $\ptl\Sg$. Most of the time we will omit the map $\var_0$. We will also identify any set $S\subseteq\Sg$ with $\var_0(S)$, and the tangent space $T_p\Sg$ at a point $p\in\Sg$ with $(d\var_0)_p(T_p\Sg)$. For a fixed smooth unit normal vector field $N$ on $\Sg$, the associated \emph{shape operator} at $p$ is the endomorphism $B_p:T_p\Sg\to T_p\Sg$ introduced by $B_p(w):=-D_wN$. Here $D$ denotes the Levi-Civita connection for the Euclidean metric. It is well known that $B_p$ is self-adjoint with respect to the metric induced by the scalar product in $\rrn$.

The \emph{anisotropic Gauss map} or \emph{anisotropic normal} in $\Sg$ is the map $N_K:\Sg\to\ptl K$ given by
\begin{equation}
\label{eq:normal} 
N_K:=\pi_K\circ N.
\end{equation} 
By setting $\var_K:=h_K(N)$ we infer from \eqref{eq:gradhK} that
\begin{equation}
\label{eq:nkn}
\escpr{N_K,N}=\var_K\quad\text{on }  \Sg. 
\end{equation}
When $\ptl\Sg\neq\emptyset$ we introduce the \emph{anisotropic conormal} by equality
\begin{equation}
\label{eq:conormal}
\nu_K:=\var_K\,\nu-\escpr{N_K,\nu}\,N,
\end{equation}
where $\nu$ is the \emph{inner conormal} along $\ptl\Sg$. Note that $\nu_K$ is a normal vector to $\ptl\Sg$ with $\escpr{N_K,\nu_K}=0$.

For any $p\in\Sg$ we have $T_{N_K(p)}(\ptl K)=T_p\Sg$, so that the differential $(dN_K)_p$ is an endomorphism of $T_p\Sg$. The \emph{anisotropic shape operator} at $p$ is defined by
\begin{equation}
\label{eq:shape}
(B_K)_p:=-(dN_K)_p=(d\pi_K)_{N(p)}\circ B_p.
\end{equation}
The \emph{anisotropic mean curvature} at $p$ is the number
\begin{equation}
\label{eq:mc}
H_K(p):=\frac{\text{tr}((B_K)_p)}{n}=-\frac{(\divv_\Sg N_K)(p)}{n},
\end{equation}
where $\divv_\Sg$ is the divergence relative to $\Sg$ and $\text{tr}(f)$ is the trace of an endomorphism $f:T_p\Sg\to T_p\Sg$. It is known, see for instance Palmer~\cite[p.~3666]{palmer}, that
\begin{equation}
\label{eq:ineq}
\text{tr}\big((B_K)_p^2\big)\geq nH_K(p)^2, \quad\text{for any } p\in\Sg,
\end{equation}
and equality holds if and only $p$ is an \emph{anisotropic umbilical point}, i.e., $(B_K)_p$ is a multiple of the identity map in $T_p\Sg$. This fact has a short proof that we reproduce here for the sake of completeness. 
Since $\nabla h_K=\pi_K$ in $\rrn_*$ and $K$ is a strictly convex body, then $(d\pi_K)_{N(p)}$ can be represented as a positive definite symmetric matrix of order $n$. Thus, there exist an orthonormal basis $\{e_1,\ldots,e_n\}$ of $T_p\Sg$ and $\{\la_1,\ldots,\la_n\}\subset\rr^+$ such that $(d\pi_K)_{N(p)}(e_j)=\la_j\,e_j$ for any $j=1,\ldots,n$. From \eqref{eq:shape} we get $(B_K)_p(e_i)=\sum_{j=1}^n\sg_{ij}\,\la_j\,e_j$, where $\sg_{ij}:=\escpr{B_p(e_i),e_j}$. By using the Cauchy-Schwarz inequality in $\rr^n$ with the vectors $(\sg_{11}\la_1,\ldots,\sg_{nn}\la_n)$ and $(1,\ldots,1)$, we obtain
\begin{align*}
\text{tr}(B_K^2)-nH_K^2=\sum_{i,j=1}^n\sg^2_{ij}\,\la_i\,\la_j-\frac{1}{n}\,\left(\sum_{i=1}^n\sg_{ii}\,\la_i\right)^2\geq\sum_{i,j=1}^n\sg^2_{ij}\,\la_i\,\la_j-\sum_{i=1}^n\sg^2_{ii}\,\la^2_i=\sum_{i\neq j}\sg^2_{ij}\,\la_i\,\la_j\geq 0.
\end{align*}
Equality holds if and only if $\sg_{ij}=0$ for any $i\neq j$ and $\sg_{ii}\,\la_i=\sg_{jj}\,\la_j$ for any $i,j=1,\ldots,n$. By denoting $\alpha:=\sg_{ii}\,\la_i$, this is equivalent to that $(B_K)_p(w)=\alpha\,w$ for any $w\in T_p\Sg$, as desired.

\begin{examples}
\label{ex:umbilical}
(i). If $K$ is the round unit ball about $0$, then $N_K=N$, $B_K=B$ and $H_K$ equals the Euclidean mean curvature of $\Sg$.

(ii). For a hyperplane $\Sg\subset\rrn$ oriented with a unit normal $N$, the map $N_K$ is constant, so that $B_K=0$ and $H_K=0$ in $\Sg$.

(iii). Consider the hypersurface $\Sg=\ptl K$ with outer unit normal $N=\eta_K$. Then $N_K(p)=p$, $(B_K)_p(w)=-w$ and $H_K(p)=-1$, for any $p\in\Sg$ and $w\in T_p\Sg$. Observe that, if $K$ is not centrally symmetric about $0$, and we take the inner normal $N=-\eta_K$ on $\Sg$, then similar equalities need not hold and $H_K$ may be nonconstant. So, the role of the chosen unit normal $N$ is very important for the computations.
\end{examples}

The previous examples show that $\ptl K$ and Euclidean hyperplanes are \emph{anisotropic umbilical hypersurfaces}, i.e., all of its points are anisotropic umbilical. Next, we provide a converse statement after Palmer~\cite[p.~3666]{palmer} and Clarenz~\cite[p.~358]{clarenz-2004}.

\begin{proposition}
\label{prop:umbilical}
Let $\Sg$ be a two-sided, connected, anisotropic umbilical hypersurface immersed in $\rrn$, and such that $H_K$ is constant. If $H_K=0$ then $\Sg$ is contained inside a hyperplane. If $H_K\neq 0$ then, up to a translation and a dilation centered at $0$, we have $\Sg\subset\ptl K$.
\end{proposition}

\begin{proof}
For any $p\in\Sg$, there is $\alpha(p)\in\rr$ such that $(B_K)_p(w)=\alpha(p)\,w$ for any $w\in T_p\Sg$. By definition~\eqref{eq:mc} it is clear that $\alpha(p)=H_K$ for any $p\in\Sg$. By equation~\eqref{eq:shape} we deduce that the differential of $N_K+H_K\,\text{Id}$ vanishes on $\Sg$. In case $H_K=0$ this implies that $N_K$ is a constant vector $n$, and so $\Sg$ is within a hyperplane with unit normal $\eta_K(n)$. When $H_K\neq 0$ we can find $c\in\rrn$ satisfying
\[
p=\frac{c}{H_K}-\frac{1}{H_K}\,N_K(p), \quad \text{for any }p\in\Sg.
\]
From here we infer that $\Sg\subseteq (c/H_K)-(1/H_K)\,(\ptl K)$, as we claimed.
\end{proof}

We finish this section with the notions of  \emph{anisotropic area} and \emph{algebraic volume} for a two-sided hypersurface $\Sg$ with unit normal $N$. If $d\Sg$ denotes the area element of $\Sg$, then $A_K(\Sg)$ is given by
\begin{equation}
\label{eq:area}
A_K(\Sg):=\int_\Sg\var_K\,d\Sg=\int_\Sg h_K(N)\,d\Sg.
\end{equation}
This coincides with the Euclidean area of $\Sg$ when $K$ is the round unit ball about $0$. When $K$ is not centrally symmetric about $0$, the value of $A_K(\Sg)$ may depend on the normal $N$ over $\Sg$. On the other hand, for a compact hypersurface $\Sg$, we follow Barbosa and do Carmo~\cite[Eq.~(2.2)]{bdc} to define
\begin{equation}
\label{eq:volume}
V(\Sg):=\frac{1}{n+1}\int_{\Sg}\escpr{p,N(p)}\,d\Sg.
\end{equation}
For $\Sg$ embedded and small enough, the divergence theorem ensures that $|V(\Sg)|$ equals the Lebesgue measure of the cone over $\Sg$ with vertex at $0$. If $\Sg$ is immersed then it is possible that $V(\Sg)=0$.

Consider a dilation $\delta_\la(p):=\la\,p$ with $\la>0$ and $p\in\rrn$. A unit normal $N_\la$ on the hypersurface $\delta_\la(\Sg)$ is determined by $N_\la(\delta_\la(p)):=N(p)$ for any $p\in\Sg$. Therefore, the change of variables formula and the fact that the Jacobian of the diffeomorphism $\delta_{\la|\Sg}:\Sg\to\delta_\la(\Sg)$ equals $\la^n$, entail that
\begin{equation}
\label{eq:dilarea}
\begin{split}
A_K(\delta_\la(\Sg))&=\la^n\,A_K(\Sg),
\\
V(\delta_\la(\Sg))&=\la^{n+1}\,V(\Sg).
\end{split}
\end{equation}
These identities will play a relevant role in the proof of our main result in Section~\ref{sec:main}.

\begin{remark}
\label{re:other}
It is usual to introduce the anisotropic area $A_F$ associated to an arbitrary function $F:\sph^n\to\rr^+$. This has the expression $A_F(\Sg):=\int_\Sg F(N)\,d\Sg$. Convexity assumptions on $F$ are necessary to deduce fine properties for $A_F$ and for the $F$-Laplacian, see for instance Maggi~\cite[Ch.~20]{maggi} or Palmer~\cite{palmer}. By extending $F$ as a 1-homogeneous function we get an asymmetric norm $\Phi$ in $\rrn$. The \emph{Wulff shape} for $\Phi$ is the convex body $\ww_\Phi$ supported by $\Phi$, see \cite[Eq.~(20.8)]{maggi}. As we remembered in the Introduction, $\ww_\Phi$ minimizes the anisotropic perimeter among sets of the same Euclidean volume. In our context we have $F=h_{K|\sph^n}$, $\Phi=h_K$ and $\ww_\Phi=K$. So, the initial convex body $K$ is the optimal shape for the anisotropic area $A_K$ defined from its support function $h_K$. In Palmer~\cite[Sect.~1]{palmer} the \emph{Wulff hypersurface} of $F$ is the strictly convex hypersurface defined as $\phi(\sph^n)$, where $\phi(w):=F(w)\,w+(\nabla_{\sph^n}F)(w)$. When $F=h_{K|\sph^n}$ it follows from equation~\eqref{eq:gradhK} that $\phi=\pi_K$ on $\sph^n$, so that the corresponding Wulff hypersurface $\phi(\sph^n)$ equals $\ptl K$. 
\end{remark}

\section{A second variation formula for the anisotropic area}
\label{sec:varform}
\noindent

In this section we compute the second derivative of the anisotropic area for certain deformations of an immersed hypersurface with non-empty boundary. The resulting formula will be employed in the proof of our main result in Theorem~\ref{th:main}. We begin with some preliminary definitions. 

A \emph{flow $($of diffeomorphisms$)$} in $\rrn$ is a smooth map $\phi:\rrn\times\rr\to\rrn$ such that $\phi(p,0)=p$ for any $p\in\rrn$, and the map $\phi_t:\rrn\to\rrn$ defined by $\phi_t(p):=\phi(p,t)$ is a diffeomorphism for any $t\in\rr$. Usually we will denote a flow by $\{\phi_t\}_{t\in\rr}$. The associated \emph{velocity vector field} is given by
\[
X(p):=\frac{d}{dt}\bigg|_{t=0}\phi_t(p),\quad\text{for any }p\in\rrn.
\]
When $Y$ is a smooth complete vector field on $\rrn$, the corresponding one-parameter group of diffeomorphisms $\{\phi_t\}_{t\in\rr}$ is a flow with velocity $Y$. 

For a hypersurface $\Sg$ immersed in $\rrn$ the \emph{variation of $\Sg$} induced by a flow $\{\phi_t\}_{t\in\rr}$ is the family $\{\Sg_t\}_{t\in\rr}$, where $\Sg_t:=\phi_t(\Sg)$ for any $t\in\rr$. The flow is \emph{compactly supported on} $\Sg$ if there is a compact set $C\subseteq\Sg$ such that $\phi_t(p)=p$ for any $p\in\Sg\setminus C$ and $t\in\rr$. If $\Sg$ is a two-sided hypersurface with unit normal $N$ then, along the variation $\{\Sg_t\}_{t\in\rr}$, we can find a smooth vector field $\ovn$ whose restriction to $\Sg_t$ provides a unit normal $N_t$ with $N_0=N$. For a fixed smooth strictly convex body $K\subset\rrn$ with support function $h_K$, the \emph{anisotropic area functional} for the variation $\{\Sg_t\}_{t\in\rr}$ is the function
\begin{equation}
\label{eq:areafunct}
A_K(t):=A_K(\Sg_t)=\int_{\Sg_t}h_K(N_t)\,d\Sg_t=\int_\Sg h_K(\ovn\circ\phi_t)\,\text{Jac}\,\phi_t\,d\Sg,
\end{equation}
where $\text{Jac}\,\phi_t$ is the Jacobian of the diffeomorphism $\phi_{t|\Sg}:\Sg\to\Sg_t$.

Our objective is to provide a useful expression of $A_K''(0)$ for some variations. For that we must compute $A_K'(0)$ first. We need some preliminary calculations that we gather below.

\begin{lemma}
\label{lem:auxiliar}
Let $K\subset\rrn$ be a smooth strictly convex body with support function $h_K$. For a two-sided hypersurface $\Sg$ immersed in $\rrn$ with unit normal $N$, the function $\var_K:=h_K(N)$ satisfies
\begin{equation}
\label{eq:gradvarK}
\nabla_\Sg\var_K=-B(N_K^\top),
\end{equation}
where $\nabla_\Sg$ is the gradient relative to $\Sg$ and $N_K^\top$ is the tangent projection of the anisotropic normal $N_K$. Moreover, for any smooth vector field $U$ with compact support on $\Sg$ and normal component $u:=\escpr{U,N}$, we have
\begin{equation}
\label{eq:divergence}
\int_\Sg\divv_\Sg(\var_K\,U)\,d\Sg=-\int_\Sg nH_K\,u\,d\Sg+\int_\Sg\escpr{N_K,\nabla_\Sg u}\,d\Sg-\int_{\ptl\Sg}\escpr{U,\nu_K}\,d(\ptl\Sg),
\end{equation}
where $H_K$ is the anisotropic mean curvature and $\nu_K$ is the anisotropic conormal.
\end{lemma}

\begin{proof}
Take a point $p\in\Sg$ and a vector $w\in T_p\Sg$. By using \eqref{eq:gradhK}, \eqref{eq:normal} and that the shape operator $B_p$ is a self-adjoint endomorphism of $T_p\Sg$, we obtain 
\begin{equation}
\label{eq:topacio}
\begin{split}
\escpr{(\nabla_\Sg\var_K)(p),w}&=\escpr{(\nabla h_K)(N(p)),D_wN}=-\escpr{\pi_K(N(p)),B_p(w)}
\\
&=-\escpr{N_K(p),B_p(w)}=-\escpr{B_p(N^\top_K(p)),w}.
\end{split}
\end{equation}
This implies \eqref{eq:gradvarK}. Let us prove \eqref{eq:divergence}. For any $w\in\rrn$ we denote by $w^\top$ and $w^\bot$ the projections of $w$ with respect to $T_p\Sg$ and $(T_p\Sg)^\bot$, respectively. For a smooth vector field $U$ on $\Sg$, note that
\[
\var_K\,U=\var_K\,U^\top+\var_K\,u\,N=\var_K\,U^\top+u\,N_K^\bot=\var_K\,U^\top+u\,N_K-u\,N_K^\top,
\]
where in the second equality we have employed \eqref{eq:nkn}. By taking divergences relative to $\Sg$ and having in mind \eqref{eq:mc}, we get
\[
\divv_\Sg(\var_K\,U)=\divv_\Sg(\var_K\,U^\top)-nH_K\,u+\escpr{N_K,\nabla_\Sg u}-\divv_\Sg(u\,N_K^\top).
\]
The desired formula follows by using \eqref{eq:conormal} and the divergence theorem on $\Sg$. 
\end{proof}

Next, we compute the first variation of $A_K$. This was previously derived by many authors, see for instance Clarenz~\cite[Sect.~1]{clarenz-2002}, or Koiso and Palmer~\cite[Proof of Prop.~3.1]{koiso-palmer2}. As in Koiso~\cite[Lem.~9.1]{koiso-2019}, our formula holds for arbitrary deformations of a Euclidean hypersurface with non-empty boundary. We include a short proof that will be helpful in the subsequent calculus of $A_K''(0)$.

\begin{proposition}
\label{prop:1starea}
Let $K\subset\rrn$ be a smooth strictly convex body, $\Sg$ a two-sided immersed hypersurface with boundary, and $\{\phi_t\}_{t\in\rr}$ a flow in $\rrn$ with compact support on $\Sg$. Then, we have
\[
A_K'(0)=-\int_\Sg nH_K\,u\,d\Sg-\int_{\ptl\Sg}\escpr{X,\nu_K}\,d(\ptl\Sg),
\]
where $H_K$ is the anisotropic mean curvature, $u:=\escpr{X,N}$ is the normal component of the velocity vector field $X$, and $\nu_K$ is the anisotropic conormal along $\ptl\Sg$.
\end{proposition}

\begin{proof}
Recall that we denote $\var_K:=h_K(N)$. For a fixed point $p\in\Sg$, we define
\begin{equation}
\label{eq:hyj}
h_p(t):=h_K(N_t\circ\phi_t)(p), \quad j_p(t):=(\text{Jac}\,\phi_t)(p), \quad\text{for any }t\in\rr.
\end{equation}
By differentiating under the integral sign in \eqref{eq:areafunct} and taking into account that $\phi_0=\text{Id}$, we obtain
\begin{equation}
\label{eq:1staK}
A_K'(0)=\int_\Sg\big(h_p'(0)+\var_K(p)\,j_p'(0)\big)\,d\Sg.
\end{equation}
On the one hand it is well known, see Simon~\cite[\S9]{simon}, that
\begin{equation}
\label{eq:1stj}
j_p'(0)=\frac{d}{dt}\bigg|_{t=0}(\text{Jac}\,\phi_t)(p)=(\divv_\Sg X)(p).
\end{equation}
On the other hand, from equations~\eqref{eq:gradhK} and \eqref{eq:normal}, it follows that
\begin{align}
h_p'(0)=\escpr{(\nabla h_K)(N(p)),D_{X(p)}\ovn}=\escpr{N_K(p),D_{X(p)}\ovn}.
\end{align}
The computation of $D_{X(p)}\ovn$ is found in \cite[Lem.~4.1(1)]{ros-souam}. We get
\begin{equation}
\label{eq:dxn}
D_{X(p)}\ovn=\frac{d}{dt}\bigg|_{t=0}(N_t\circ\phi_t)(p)=-(\nabla_\Sg u)(p)-B_p(X^\top(p)),
\end{equation}
where $X^\top(p)$ is the projection of $X(p)$ onto $T_p\Sg$. By substituting this information into the previous expression for $h_p'(0)$ and having in mind the third equality in \eqref{eq:topacio}, we arrive at
\begin{equation}
\label{eq:1sth}
h_p'(0)=-\escpr{N_K,\nabla_\Sg u}(p)+\escpr{\nabla_\Sg\var_K,X^\top}(p).
\end{equation}
Thus, the integrand in \eqref{eq:1staK} is the evaluation at $p$ of the function
\[
-\escpr{N_K,\nabla_\Sg u}+\escpr{\nabla_\Sg\var_K,X^\top}+\var_K\,\divv_\Sg X=-\escpr{N_K,\nabla_\Sg u}+\divv_\Sg(\var_K\,X).
\]
The proof finishes by applying the formula in \eqref{eq:divergence} with $U=X$.
\end{proof}

Second variation formulas for the anisotropic area under different hypotheses on $\Sg$ and the deformation can be found in Koiso and Palmer~\cite[Prop.~3.3]{koiso-palmer2}, \cite[Prop.~3.3]{koiso-palmer4}, Barbosa and Silva~\cite[Prop.~3]{barbosa-silva}, and Guo and Xia~\cite[Prop.~3.5]{guo-xia}. In this work we only need to compute $A_K''(0)$ when we move a hypersurface $\Sg$ with non-empty boundary by means of some special flows. 

\begin{proposition}
\label{prop:2ndarea}
Let $K\subset\rrn$ be a smooth strictly convex body, $\Sg$ a compact two-sided immersed hypersurface with boundary, and $X$ a smooth complete vector field on $\rrn$ such that $X_{|\Sg}=N_K$. Then, for the one-parameter group of diffeomorphisms $\{\phi_t\}_{t\in\rr}$ associated to $X$, we have
\[
A_K''(0)=\int_\Sg\left(n^2H_K^2-\emph{tr}(B_K^2)\right)\var_K\,d\Sg-\int_\Sg n H_K\,v\,d\Sg-\int_{\ptl\Sg}\escpr{Z,\nu_K}\,d(\ptl\Sg),
\]
where $H_K$ is the anisotropic mean curvature, $B_K$ is the anisotropic shape operator, $\nu_K$ is the anisotropic conormal and $v:=\escpr{Z,N}$ is the normal component of the vector field $Z:=D_XX$.
\end{proposition}

\begin{proof}
For any $w\in\rrn$ the notations $w^\top$ and $w^\bot$ will stand for the projections of $w$ onto $T\Sg$ and $(T\Sg)^\bot$, respectively. For a given point $p\in\Sg$ we define $h_p(t)$ and $j_p(t)$ as in \eqref{eq:hyj}. By differentiating under the integral sign twice in \eqref{eq:areafunct} and taking into account that $\phi_0=\text{Id}$, we get
\begin{equation}
\label{eq:2ndaK}
A_K'(0)=\int_\Sg\big(h_p''(0)+2h_p'(0)\,j_p'(0)+\var_K(p)\,j_p''(0)\big)\,d\Sg.
\end{equation}
Let us compute all the derivatives in the integrand above.

For the calculus of $j_p'(0)$ and $j_p''(0)$ we refer the reader to Simon~\cite[\S9]{simon}. From \eqref{eq:1stj}, \eqref{eq:mc}, and the fact that $X_{|\Sg}=N_K$, we obtain
\begin{equation}
\label{eq:jprima}
j_p'(0)=(\divv_\Sg X)(p)=(\divv_\Sg N_K)(p)=-nH_K(p).
\end{equation}
On the other hand, we have
\[
j_p''(0)=(\divv_\Sg Z)(p)+(\divv_\Sg X)^2(p)+\sum_{i=1}^n|(D_{e_i}X)^\bot|^2-\sum_{i,j=1}^n\escpr{D_{e_i}X,e_j}\,\escpr{D_{e_j}X,e_i},
\]
where $\{e_1,\ldots,e_n\}$ is an orthonormal basis of $T_p\Sg$. Equation~\eqref{eq:shape} and the fact that $X_{|\Sg}=N_K$ yield 
\[
D_{e_i}X=D_{e_i}N_K=-(B_K)_p(e_i),
\] 
so that $(D_{e_i}X)^\bot=0$ for any $i=1,\ldots,n$. It is also clear that
\[
\sum_{i,j=1}^n\escpr{D_{e_i}X,e_j}\,\escpr{D_{e_j}X,e_i}=
\text{tr}\big((B_K)^2_p\big).
\]
All this together with \eqref{eq:jprima} shows that
\begin{equation}
\label{eq:j2prima}
j_p''(0)=\big(\!\divv_\Sg Z+n^2H_K^2-\text{tr}(B_K^2)\big)(p).
\end{equation}

Next, we compute $h_p'(0)$ and $h_p''(0)$. Note that $\escpr{X,N}=\escpr{N_K,N}=\var_K$ by \eqref{eq:nkn}. Hence, equation~\eqref{eq:1sth} implies that
\begin{equation}
\label{eq:hprima}
h_p'(0)=-\escpr{N_K,\nabla_\Sg\var_K}(p)+\escpr{\nabla_\Sg\var_K,N_K^\top}(p)=0.
\end{equation}
In order to calculate $h_p''(0)$ we need an expression for $h_p'(t)$. From the definition in \eqref{eq:hyj} we deduce  
\[
h_p'(t)=\escpr{(\nabla h_K)((N_t\circ\phi_t)(p)),(D_{X}\ovn)(\phi_t(p))}=\escpr{\pi_K((N_t\circ\phi_t)(p)),(D_{X}\ovn)(\phi_t(p))}
\]
because of equation~\eqref{eq:gradhK}. This entails that
\[
h_p''(0)=\escpr{(d\pi_K)_{N(p)}(D_{X(p)}\ovn),D_{X(p)}\ovn}+\escpr{N_K(p),D_{X(p)}D_X\ovn}.
\]
As $\escpr{X,N}=\var_K$ on $\Sg$, equations \eqref{eq:dxn} and \eqref{eq:gradvarK} lead to
\begin{equation}
\label{eq:tente}
D_{X(p)}\ovn=-(\nabla_\Sg\var_K)(p)-B_p(N_K^\top(p))=0.
\end{equation}
On the other hand, by Lemma~\ref{lem:potato} below we know that
\[
D_{X(p)}D_X\ovn=-(\nabla_\Sg v)(p)-B_p(Z^\top(p)),
\]
and so
\begin{equation}
\label{eq:h2prima}
h_p''(0)=-\escpr{N_K,\nabla_\Sg v}(p)-\escpr{N_K(p),B_p(Z^\top(p))}
=-\escpr{N_K,\nabla_\Sg v}(p)+\escpr{\nabla_\Sg\var_K,Z^\top}(p),
\end{equation}
where we have employed the third equality in equation~\eqref{eq:topacio}.

Now, by having in mind \eqref{eq:h2prima}, \eqref{eq:hprima} and \eqref{eq:j2prima}, we conclude that the integrand in \eqref{eq:2ndaK} is the evaluation at $p$ of the function
\[
-\escpr{N_K,\nabla_\Sg v}+\divv_\Sg(\var_K\,Z)+\big(n^2H^2_K-\text{tr}(B_K^2)\big)\,\var_K.
\]
From here the the proof finishes by applying the formula in \eqref{eq:divergence} with $U=Z$.
\end{proof}

\begin{lemma}
\label{lem:potato}
In the conditions of Proposition~\ref{prop:2ndarea}, for any $p\in\Sg$, we have
\[
D_{X(p)}D_X\ovn=-(\nabla_\Sg v)(p)-B_p(Z^\top(p)).
\]
\end{lemma}

\begin{proof}
Take an orthonormal basis $\{e_1,\ldots,e_n\}$ of $T_p\Sg$. We use the flow $\{\phi_t\}_{t\in\rr}$ associated to $X$ to construct, for any $i=1,\ldots,n$, a smooth vector field $E_i$ around $p$ which is tangent on any $\Sg_t=\phi_t(\Sg)$ while satisfying $E_i(p)=e_i$ and $[X,E_i]=[\ovn,E_i]=0$ (here $[\cdot\,,\cdot]$ stands for the Lie bracket of vector fields in $\rrn$). It is clear that
\begin{equation}
\label{eq:gen0}
D_{X(p)}D_X\ovn=\sum_{i=1}^n\escpr{D_{X(p)}D_X\ovn,e_i}\,e_i
+\escpr{D_{X(p)}D_X\ovn,N(p)}\,N(p).
\end{equation}
We will compute the different terms in the previous equation.

Since $|\ovn|^2=1$ we get $\escpr{D_X\ovn,\ovn}=0$. By differentiating and applying equation \eqref{eq:tente}, we deduce
\begin{equation}
\label{eq:gen1}
0=\escpr{D_{X(p)}D_X\ovn,N(p)}+|D_{X(p)}\ovn|^2=\escpr{D_{X(p)}D_X\ovn,N(p)}.
\end{equation}
On the other hand, we differentiate twice with respect to $X$ in equality $\escpr{\ovn,E_i}=0$. By taking into account that $D_{X(p)}\ovn=0$ and $[X,E_i]=0$, we obtain
\[
\escpr{D_{X(p)}D_X\ovn,e_i}=-\escpr{N(p),D_{X(p)}D_XE_i}=-\escpr{N(p),D_{X(p)}D_{E_i}X}.
\]
Since the Riemann curvature tensor vanishes for the standard metric in $\rrn$, we infer
\[
0=D_XD_{E_i}X-D_{E_i}D_XX-D_{[X,E_i]}X=D_XD_{E_i}X-D_{E_i}Z
\]
because $D_XX=Z$. This shows that $D_{X(p)}D_{E_i}X=D_{e_i}Z$. As a consequence
\begin{align}
\escpr{D_{X(p)}D_X\ovn,e_i}&=-\escpr{N(p),D_{e_i}Z}=-e_i\big(\escpr{Z,N}\big)+\escpr{Z(p),D_{e_i}N}
\\
&=-\escpr{(\nabla_\Sg v)(p),e_i}-\escpr{Z^\top(p),B_p(e_i)}
\\
\label{eq:gen2}
&=-\escpr{(\nabla_\Sg v)(p),e_i}-\escpr{B_p(Z^\top(p)),e_i}.
\end{align}
By substituting \eqref{eq:gen1} and \eqref{eq:gen2} into \eqref{eq:gen0} the proof follows.
\end{proof}

\begin{remarks}
(i). If $K\subset\rrn$ is the round unit ball about $0$, then the anisotropic area coincides with the Euclidean area and the obtained formulas are well known.

(ii). It is interesting to observe that, unless $K$ is centrally symmetric about $0$, the formulas for $A_K'(0)$ and $A_K''(0)$ may depend on the unit normal vector $N$ fixed on $\Sg$.

(iii). In Proposition~\ref{prop:2ndareavol} we will see that, under some extra conditions, the boundary integrand in the expression of $A_K''(0)$ has a geometric interpretation.
\end{remarks}

\section{Anisotropic stable hypersurfaces in solid cones}
\label{sec:main}
\setcounter{equation}{0}

In this section we consider a Euclidean solid cone and study compact hypersurfaces immersed in the cone and minimizing the anisotropic area up to second order for deformations preserving the volume of the hypersurface and the boundary of the cone. In this situation we will show that the variational formulas computed in Section~\ref{sec:varform} can be slightly simplified. After that we will prove our main theorem, where we characterize these second order minima when the cone is convex. It is worth mentioning that, excluding this classification statement, all the results in this section still hold when we replace the cone with any smooth Euclidean open set.    

We begin by introducing some notation and definitions. For a domain $\de\subset\sph^n$ with smooth boundary, the \emph{solid cone} over $\de$ is the set
\[
\cc:=\{\la\,p\,;\, \la>0,\, p\in\de\}.
\]
This is a domain of $\rrn$ with boundary $\ptl\cc$ smooth away from $0$. We call $\xi$ to the inner unit normal along $\ptl\cc\setminus\{0\}$. Note that $\cc$ and $\ptl\cc$ are invariant under the dilations $\delta_\la$ centered at $0$. It is also clear that $\cc$ coincides with an open half-space when $\de$ is an open hemisphere. 

Let $\Sg$ be a smooth, compact, two-sided hypersurface immersed in $\overline{\cc}$ with smooth boundary $\ptl\Sg$ in $\ptl\cc\setminus\{0\}$. We suppose that $\Sg\cap\ptl\cc=\ptl\Sg$, so that $\Sg\setminus\ptl\Sg\subseteq\cc$. We fix a smooth unit normal vector field $N$ on $\Sg$. The inner conormal vector of $\ptl\Sg$ in $\Sg$ is represented by $\nu$. In the planar distribution $T(\ptl\Sg)^\bot$ we choose the orientation induced by $\{\nu,N\}$. Thus, for any $p\in\ptl\Sg$, there is a unique $\mu(p)\in T_p(\ptl\Sg)^\bot$ such that $\{\xi(p),\mu(p)\}$ is a positively oriented orthonormal basis. Observe that $\mu$ is tangent to $\ptl\cc$ and normal to $\ptl\Sg$. It is easy to check that these equalities hold along $\ptl\Sg$
\begin{equation}
\label{eq:angles}
\begin{split}
\nu&=(\cos\theta)\,\xi-(\sin\theta)\,\mu, \qquad \mu=-(\sin\theta)\,\nu+(\cos\theta)\,N,
\\
N&=(\sin\theta)\,\xi+(\cos\theta)\,\mu,\qquad \xi=(\cos\theta)\,\nu+(\sin\theta)\,N,
\end{split}
\end{equation}
where $\theta$ is the oriented angle function between $\nu$ and $\xi$ in $T(\ptl\Sg)^\bot$. As a consequence, for a fixed smooth strictly convex body $K\subset\rrn$, the anisotropic normal $N_K$ on $\Sg$ and the anisotropic conormal $\nu_K$ along $\ptl\Sg$ given in \eqref{eq:normal} and \eqref{eq:conormal} verify
\begin{equation}
\label{eq:anangles}
\escpr{N_K,\mu}=\escpr{\nu_K,\xi},\qquad \escpr{N_K,\xi}=-\escpr{\nu_K,\mu}.
\end{equation}

A flow of diffeomorphisms $\{\phi_t\}_{t\in\rr}$ in $\rrn$ is \emph{admissible} for $\cc$ if $\phi_t(\ptl\cc)=\ptl\cc$ and $\phi_t(0)=0$, for any $t\in\rr$. The induced variation $\{\Sg_t\}_{t\in\rr}$ of $\Sg$ satisfies $\ptl\Sg_t\subset\ptl\cc\setminus\{0\}$ and $\Sg_t\cap\ptl\cc=\ptl\Sg_t$, for any $t\in\rr$. Moreover, the velocity vector field $X$ is tangent on $\ptl\cc\setminus\{0\}$ and vanishes at $0$. We consider the \emph{anisotropic area functional} $A_K(t)$ introduced in \eqref{eq:areafunct} and the \emph{volume functional} $V(t)$, which assigns to any $t\in\rr$ the \emph{algebraic volume enclosed} by $\Sg_t$ as defined in \eqref{eq:volume}. We computed $A_K'(0)$ in Proposition~\ref{prop:1starea}. On the other hand, 
it is well known (\cite[Eq.~(2.3)]{bdc}) that
\begin{equation}
\label{eq:1stvol}
V'(0)=\int_\Sg u\,d\Sg,
\end{equation}
where $u:=\escpr{X,N}$. We say that the flow preserves the volume of $\Sg$ if $V(t)$ is constant for any $t$ small enough. This implies that $\int_\Sg u\,d\Sg=0$. Conversely, for any smooth function $u:\Sg\to\rr$ with $\int_\Sg u\,d\Sg=0$, there is a flow $\{\phi_t\}_{t\in\rr}$ admissible for $\cc$, preserving the volume of $\Sg$, and such that $\escpr{X,N}=u$ on $\Sg$, see Barbosa and do Carmo~\cite[Lem.~(2.4)]{bdc}. 

A hypersurface $\Sg$ in the previous conditions is \emph{anisotropic stationary} if $A_K'(0)=0$ for any admissible flow for $\cc$ preserving the volume of $\Sg$. The first variational formulas in Proposition~\ref{prop:1starea} and equation \eqref{eq:1stvol}, together with the aforementioned construction of volume-preserving flows, lead to a characterization of anisotropic stationary hypersurfaces. Indeed, we can reason as Koiso and Palmer~\cite[Prop.~3.1]{koiso-palmer2} to deduce the next result, which generalizes the isotropic situation.

\begin{proposition}
\label{prop:varprop}
A two-sided hypersurface $\Sg$ immersed in $\overline{\cc}$ with boundary $\ptl\Sg$ in $\ptl\cc\setminus\{0\}$ is anisotropic stationary if and only if the anisotropic mean curvature $H_K$ is constant on $\Sg$ and $\escpr{N_K,\xi}=0$ along $\ptl\Sg$.
\end{proposition}

\begin{example}
\label{ex:typical}
Take the hypersurface $\Sg=\ptl K\cap\overline{\cc}$ with unit normal $N=\eta_K$ (the one pointing outside $K$). Note that $\ptl K$ meets $\ptl\cc$ transversally because $K$ is convex and contains $0$ in its interior. Hence, $\Sg$ is a compact hypersurface with boundary $\ptl\Sg$ such that $\Sg\cap\ptl\cc=\ptl\Sg$. In Examples~\ref{ex:umbilical} (iii) we saw that $H_K(p)=-1$ and $N_K(p)=p$, for any $p\in\Sg$. As $\ptl\cc$ is invariant under dilations centered at $0$ then $N_K(p)\in T_p(\ptl\cc)$ for any $p\in\ptl\Sg$ and so, $\escpr{N_K,\xi}=0$. From the previous proposition we conclude that $\Sg$ is anisotropic stationary.
\end{example}

In the isotropic case the orthogonality condition for a stationary hypersurface $\Sg$ entails that $\nu=\xi$ along $\ptl\Sg$. In the next lemma we establish a similar equality for the anisotropic conormal $\nu_K$ that will be useful in future results.

\begin{lemma}
\label{lem:nuK}
Let $\Sg$ be a two-sided hypersurface immersed in $\overline{\cc}$ with boundary $\ptl\Sg$ in $\ptl\cc\setminus\{0\}$. If $\escpr{N_K,\xi}=0$ in $\ptl\Sg$, then $\cos\theta$ never vanishes and
\[
\nu_K=\frac{\var_K}{\cos\theta}\,\xi\quad\text{along }\ptl\Sg,
\]
where $\var_K:=h_K(N)$ and $\theta$ is the oriented angle function between $\nu$ and $\xi$ in $T(\ptl\Sg)^\bot$.
\end{lemma}

\begin{proof}
Recall that $\nu_K:=\var_K\,\nu-\escpr{N_K,\nu}\,N$, which is normal to $\ptl\Sg$. By equation~\eqref{eq:anangles} we have $\escpr{\nu_K,\mu}=-\escpr{N_K,\xi}=0$. This shows that $\nu_K$ is normal to $\ptl\cc$ and so, $\nu_K=\escpr{\nu_K,\xi}\,\xi$ along $\ptl\Sg$. 

Let us compute $\escpr{\nu_K,\xi}$. Consider the projection $N_K^*$ of $N_K$ onto $T(\ptl\Sg)^\bot$. By using \eqref{eq:anangles}, equality $\escpr{N_K,\xi}=0$, and that $\{\nu,N\}$ and $\{\xi,\mu\}$ are orthonormal basis of $T(\ptl\Sg)^\bot$, we obtain
\[
\escpr{\nu_K,\xi}^2=\escpr{N_K,\mu}^2=|N_K^*|^2=\escpr{N_K,\nu}^2+\var_K^2,
\]
which is a positive number. Hence, $\escpr{\nu_K,\xi}$ never vanishes along $\ptl\Sg$. On the other hand, by substituting the expression for $\xi$ in \eqref{eq:angles} into equality $\escpr{N_K,\xi}=0$, we get
\[
\sin\theta=-\frac{\escpr{N_K,\nu}\,\cos\theta}{\var_K}\quad\text{along }\ptl\Sg.
\]
The definition of $\nu_K$ and the two previous relations lead to
\[
\escpr{\nu_K,\xi}=\var_K\,\cos\theta-\escpr{N_K,\nu}\,\sin\theta=\frac{\var^2_K+\escpr{N_K,\nu}^2}{\var_K}\,\cos\theta=\frac{\escpr{\nu_K,\xi}^2}{\var_K}\,\cos\theta,
\]
and so
\[
\frac{\escpr{\nu_K,\xi}}{\var_K}\,\cos\theta=1\quad\text{along }\ptl\Sg.
\] 
This implies that $\cos\theta$ never vanishes and allows to deduce the announced expression for $\nu_K$. 
\end{proof}

Now, we can provide new expressions for the derivatives of the anisotropic area when we consider an anisotropic stationary hypersurface $\Sg$ in $\cc$. On the one hand, the boundary term in the formula for $A_K'(0)$ obtained in Proposition~\ref{prop:1starea} vanishes for any flow $\{\phi_t\}_{t\in\rr}$ admissible for $\cc$. This comes from Proposition~\ref{prop:varprop} and Lemma~\ref{lem:nuK} since the velocity vector field $X$ is tangent to $\ptl\cc\setminus\{0\}$. Hence
\begin{equation}
\label{eq:1stareast}
A_K'(0)=-nH_K\int_\Sg u\,d\Sg,
\end{equation}
where $u:=\escpr{X,N}$ on $\Sg$. Note that, for the flow $\phi_t(p):=e^t\,p$, the first equation in \eqref{eq:dilarea} implies that $A_K'(0)=nA_K(\Sg)$. Thus, by having in mind \eqref{eq:1stareast} and equation~\eqref{eq:volume}, it follows that
\begin{equation}
\label{eq:minkowski}
A_K(\Sg)=-(n+1)\,H_K\,V(\Sg).
\end{equation}
This identity is a \emph{Minkowski-type formula} for compact anisotropic stationary hypersurfaces in $\cc$. 

On the other hand, the boundary integrand appearing in $A_K''(0)$, see Proposition~\ref{prop:2ndarea}, can be written in terms of the Euclidean extrinsic geometry of $\ptl\cc$. We prove this in the next proposition, where we also compute the second derivative of the volume for certain flows.

\begin{proposition}
\label{prop:2ndareavol}
Let $\Sg$ be a two-sided anisotropic stationary hypersurface immersed in $\overline{\cc}$ with boundary $\ptl\Sg$ in $\ptl\cc\setminus\{0\}$. Consider a smooth complete vector field $X$ on $\rrn$ such that $X_{|\Sg}=N_K$, $X(0)=0$ and $X$ is tangent to $\ptl\cc\setminus\{0\}$. Then, for the admissible flow $\{\phi_t\}_{t\in\rr}$ defined by the one-parameter group of diffeomorphisms associated to $X$, we have
\begin{align}
\label{eq:2ndareast}
A_K''(0)&=\int_\Sg\left(n^2H_K^2-\emph{tr}(B_K^2)\right)\var_K\,d\Sg-\int_\Sg n H_K\,v\,d\Sg-\int_{\ptl\Sg}\frac{\emph{II}(N_K,N_K)}{\cos\theta}\,\var_K\,d(\ptl\Sg),
\\
\label{eq:2ndvol}
V''(0)&=\int_\Sg(-nH_K\,\var_K+v)\,d\Sg,
\end{align}
where $B_K$ is the anisotropic shape operator of $\Sg$, $\emph{II}$ is the second fundamental form of $\ptl\cc\setminus\{0\}$ with respect to the inner normal, $\theta$ is the oriented angle function between $\nu$ and $\xi$ in $T(\ptl\Sg)^\bot$, and $v:=\escpr{Z,N}$ is the normal component of $Z:=D_XX$.
\end{proposition}

\begin{remark}
Observe that $N_K$ is tangent to $\ptl\cc$ along $\ptl\Sg$ because $\Sg$ is anisotropic stationary. This guarantees the existence of a vector field $X$ in the conditions of the statement and that the term $\text{II}(N_K,N_K)$ is well defined. Note also that $\cos\theta$ never vanishes by Lemma~\ref{lem:nuK}.
\end{remark}

\begin{proof}[Proof of Proposition~\ref{prop:2ndareavol}]
We first check that \eqref{eq:2ndareast} holds. From Proposition~\ref{prop:2ndarea} it suffices to see that
\[
\escpr{Z,\nu_K}=\frac{\text{II}(N_K,N_K)}{\cos\theta}\,\var_K\quad\text{along }\ptl\Sg.
\]
By Lemma~\ref{lem:nuK} we know that
\[
\escpr{Z,\nu_K}=\frac{\var_K}{\cos\theta}\,\escpr{Z,\xi}\quad\text{along }\ptl\Sg.
\]
It is clear that $\escpr{X,\xi}=0$ along $\ptl\cc$ because $X$ is tangent to $\ptl\cc\setminus\{0\}$ and $X(0)=0$. By differentiating with respect to $X$, we obtain
\[
0=\escpr{D_XX,\xi}+\escpr{X,D_X\xi}=\escpr{Z,\xi}-\text{II}(X,X),
\]
so that $\escpr{Z,\xi}=\text{II}(X,X)=\text{II}(N_K,N_K)$ along $\ptl\Sg$.

Now we compute $V''(0)$. Equation~\eqref{eq:1stvol} gives us
\[
V'(t)=\int_{\Sg_t}\escpr{X,N_t}\,d\Sg_t=\int_\Sg\big(\escpr{X,N_t}\circ\phi_t\big)\,\text{Jac}\,\phi_t\,d\Sg.
\]
By differentiating and having in mind \eqref{eq:dxn}, \eqref{eq:1stj}, \eqref{eq:tente}, \eqref{eq:mc} and \eqref{eq:nkn}, it follows that
\[
V''(0)=\int_\Sg\big(\escpr{D_XX,N}+\escpr{X,D_X\overline{N}}+\escpr{X,N}\,\divv_\Sg X\big)\,d\Sg=\int_\Sg(v-nH_K\,\var_K)\,d\Sg.
\]
This completes the proof.
\end{proof}

We now turn to the main result of the paper. This is a classification of anisotropic stable hypersurfaces in a Euclidean solid cone. A two-sided hypersurface $\Sg$ immersed in $\overline{\cc}$ with boundary $\ptl\Sg$ in $\ptl\cc\setminus\{0\}$ is \emph{anisotropic stable} if $A_K'(0)=0$ and $A_K''(0)\geq 0$ for any flow admissible for $\cc$ and preserving the volume of $\Sg$. When the cone is convex, the hypersurface $\ptl K\cap\overline{\cc}$ is anisotropic stable since it 
minimizes the anisotropic area (computed with respect to the outer unit normal) among compact hypersurfaces in $\cc$ separating the same volume, see \cite[Cor.~1.2]{milman-rotem} and \cite[Thm.~1.3]{cabre-rosoton-serra}. So, it is natural to ask if this property characterizes $\ptl K\cap\overline{\cc}$ up to translations and dilations centered at $0$. The next example shows that, in general, the answer is negative.

\begin{example}
\label{ex:nonuniqueness}
Let $\cc$ be a convex cone different from a Euclidean half-space and such that $\ptl\cc$ contains a half-hyperplane $P$ (for instance, any solid cone in $\rr^2$ satisfies this property). As we observed above, if $P^+$ is the open half-space with $\cc\subset P^+$, then $\ptl K\cap\overline{P^+}$ minimizes the anisotropic area in $P^+$ for fixed volume. In particular, by applying a suitable translation along $P$ we produce an anisotropic stable hypersurface in $\cc$ which has not the form $p_0+\la\,(\ptl K\cap\overline{\cc})$ for some $p_0\in\rrn$ and $\la>0$.
\end{example}

This example illustrates that the optimal conclusion to be deduced for an anisotropic stable hypersurface $\Sg$ in $\cc$ is that $\Sg\subset\ptl K$, up to translation and homothety. We prove this fact in the next uniqueness statement under our regularity conditions for $\Sg$, $\cc$ and $K$. 

\begin{theorem}
\label{th:main}
Let $\cc\subset\rrn$ be a solid convex cone over a smooth domain of $\sph^n$. Consider a compact, connected, two-sided hypersurface $\Sg$ immersed in $\overline{\cc}$ with smooth boundary $\ptl\Sg$ in $\ptl\cc\setminus\{0\}$ and such that $\Sg\cap\ptl\cc=\ptl\Sg$. If $\Sg$ is anisotropic stable for the area $A_K$ defined by a smooth strictly convex body $K\subset\rrn$ then, there is $p_0\in\rrn$ and $\la>0$ such that $\Sg\subset p_0+\la\,(\ptl K)$.
\end{theorem}

\begin{proof}
We will follow the idea explained in the Introduction. As $\Sg$ is anisotropic stationary, Proposition~\ref{prop:varprop} implies that $H_K$ is constant on $\Sg$ and $\escpr{N_K,\xi}=0$ along $\ptl\Sg$. So, we can find a smooth complete vector field $X$ on $\rr^n$ such that $X_{|\Sg}=N_K$, $X(0)=0$ and $X$ is tangent to $\ptl\cc\setminus\{0\}$. It is clear that $u=\var_K$ on $\Sg$, where $u:=\escpr{X,N}$ and $\var_K:=h_K(N)$. Let $\{\psi_t\}_{t\in\rr}$ be the one-parameter group of diffeomorphisms associated to $X$. We take the functionals $A_K(t):=A_K(\Sg_t)$ and $V(t):=V(\Sg_t)$ associated to the variation $\Sg_t:=\psi_t(\Sg)$. From the Minkowski formula in equation~\eqref{eq:minkowski} we have $H_K\neq 0$ and $V(\Sg)\neq 0$. Hence, there is $\eps>0$ such that $V(t)$ has the same sign as $V(\Sg)$ for any $t\in(-\eps,\eps)$. Next, we apply to $\Sg_t$ a dilation $\delta_{\la(t)}(p):=\la(t)\,p$ with $\la(t)>0$, so that the volume of $\delta_{\la(t)}(\Sg_t)$ equals the volume of $\Sg$. Since $V(\delta_{\la(t)}(\Sg_t))=\la(t)^{n+1}\,V(t)$, see~\eqref{eq:dilarea}, we get
\begin{equation}
\label{eq:lat}
\la(t):=\left(\frac{V(\Sg)}{V(t)}\right)^{\frac{1}{n+1}}, \quad\text{for any } t\in(-\eps,\eps).
\end{equation}
In particular, $\la(0)=1$. We extend $\la(t)$ to a smooth positive function on $\rr$. If we define
\[ 
\phi_t:=\delta_{\la(t)}\circ\psi_t, \quad\text{for any }t\in\rr,
\]
then we produce an admissible flow for $\cc$ that preserves the volume of $\Sg$. Hence, the anisotropic stability of $\Sg$ entails that the functional $a_K(t):=A_K(\phi_t(\Sg))$ satisfies $a_K''(0)\geq 0$. To prove the theorem we need to compute $a_K''(0)$. By equation~\eqref{eq:dilarea} we know that
\begin{equation}
\label{eq:akt}
a_K(t)=A_K(\delta_{\la(t)}(\Sg_t))=\la(t)^n\,A_K(t).
\end{equation}
Thus, the calculus of $a_K''(0)$ relies on the values of $\la'(0)$, $\la''(0)$, $A_K'(0)$ and $A_K''(0)$.

By using equations~\eqref{eq:1stvol}, \eqref{eq:1stareast} and the fact that $u=\var_K$ on $\Sg$, we obtain
\begin{align}
\label{eq:prima1vol}
V'(0)&=A_K(\Sg),
\\
\label{eq:prima1area}
A_K'(0)&=-nH_K A_K(\Sg).
\end{align}
On the other hand, from the expression of $\la(t)$ in \eqref{eq:lat}, we have
\begin{equation}
\label{eq:latprima}
\la'(t)=-\frac{1}{n+1}\,V(\Sg)^{\frac{1}{n+1}}\,V(t)^{\frac{-n-2}{n+1}}\,V'(t),
\end{equation}
and so, from equation \eqref{eq:minkowski}, it follows that
\begin{equation}
\label{eq:latprima1}
\la'(0)=-\frac{1}{n+1}\,\frac{A_K(\Sg)}{V(\Sg)}=H_K.
\end{equation}

Now we compute $A_K''(0)$ and $\la''(0)$. From \eqref{eq:2ndvol} and \eqref{eq:2ndareast}, since $H_K$ is constant, we get
\begin{align}
\label{eq:prima2vol}
V''(0)&=-nH_KA_K(\Sg)+\alpha,
\\
\label{eq:prima2area}
A_K''(0)&=n^2H_K^2\,A_K(\Sg)-\int_\Sg\text{tr}(B_K^2)\,\var_K\,d\Sg-n H_K\,\alpha-\int_{\ptl\Sg}\frac{\text{II}(N_K,N_K)}{\cos\theta}\,\var_K\,d(\ptl\Sg),
\end{align}
where $\alpha:=\int_\Sg v\,d\Sg$ and $v:=\escpr{D_XX,N}$. Recall that $\text{II}$ is the second fundamental form of $\ptl\cc\setminus\{0\}$ with respect to the inner normal $\xi$. On the other hand, from \eqref{eq:latprima} we deduce
\begin{equation}
\label{eq:lat2prima}
\la''(t)=\frac{-1}{n+1}\,V(\Sg)^{\frac{1}{n+1}}\,\left(-\frac{n+2}{n+1}\,V(t)^{\frac{-2n-3}{n+1}}\,V'(t)^{2}+V(t)^{\frac{-n-2}{n+1}}\,V''(t)\right).
\end{equation}
By evaluating at $t=0$ and simplifying, equations~\eqref{eq:prima1vol} and \eqref{eq:prima2vol} give us
\[
\lambda''(0)=\frac{n+2}{(n+1)^2}\,\frac{A_K(\Sg)^2}{V(\Sg)^2}+\frac{n}{n+1}\,\frac{H_KA_K(\Sg)}{V(\Sg)}-\frac{1}{n+1}\,\frac{\alpha}{V(\Sg)}.
\]
When we employ the identity \eqref{eq:minkowski} in the three summands above, we arrive at
\begin{equation}
\label{eq:lat2prima2}
\la''(0)=2H_K^2+\frac{H_K}{A_K(\Sg)}\,\alpha.
\end{equation}

Finally, we differentiate into equation~\eqref{eq:akt} to infer
\[
a_K'(t)=n\,\la(t)^{n-1}\,\la'(t)\,A_K(t)+\la(t)^{n}\,A_K'(t).
\]
As a consequence
\[
a_K''(t)=n\left\{(n-1)\,\la(t)^{n-2}\,\la'(t)^{2}\,A_K(t)+\la(t)^{n-1}\,\la''(t)\,A_K(t)+2\la(t)^{n-1}\,\la'(t)\,A_K'(t)\right\}+\la(t)^n\,A_K''(t).
\]
By substituting above the expressions in \eqref{eq:latprima1}, \eqref{eq:lat2prima2}, \eqref{eq:prima1area}, \eqref{eq:prima2area}, and simplifying, we obtain
\begin{equation}
\label{eq:desi}
a_K''(0)=-\int_\Sg\big(\text{tr}(B_K^2)-nH_K^2\big)\,\var_K\,d\Sg-\int_{\ptl\Sg}\frac{\text{II}(N_K,N_K)}{\cos\theta}\,\var_K\,d(\ptl\Sg).
\end{equation}

The first integrand in the previous formula is nonnegative by \eqref{eq:ineq}. The convexity of $\cc$ implies that $\text{II}(N_K,N_K)\geq 0$ along $\ptl\Sg$. By the first equation in \eqref{eq:angles} we have $\cos\theta=\escpr{\nu,\xi}$, which is positive along $\ptl\Sg$ because $\cc$ is convex and $\nu$ points into $\Sg\subset\overline{\cc}$. Hence, the stability inequality $a_K''(0)\geq 0$ entails that $\text{tr}(B_K^2)=nH_K^2$ on $\Sg$ (this means that $\Sg$ is anisotropic umbilical) and $\text{II}(N_K,N_K)=0$ along $\ptl\Sg$. Since $H_K\neq 0$, the proof finishes by invoking Proposition~\ref{prop:umbilical}.
\end{proof}

The situation described in Example~\ref{ex:nonuniqueness} leads us to seek additional conditions on the cone in order to deduce stronger uniqueness conclusions. In this direction we can prove the next statement.

\begin{corollary}
\label{cor:stronger}
Let $\cc\subset\rrn$ be a solid cone over a smooth strictly convex domain of $\sph^n$. Consider a compact, connected, two-sided hypersurface $\Sg$ immersed in $\overline{\cc}$ with smooth boundary $\ptl\Sg$ in $\ptl\cc\setminus\{0\}$ and such that $\Sg\cap\ptl\cc=\ptl\Sg$. If $\Sg$ is anisotropic stable for the area $A_K$ defined by a smooth strictly convex body $K\subset\rrn$, then $\Sg=\la\,(\ptl K\cap\overline{\cc})$ for some $\la>0$.
\end{corollary}

\begin{proof}
From Theorem~\ref{th:main} we know that $\Sg\subset p_0+\la\,(\ptl K)$ for some $p_0\in\rrn$ and $\la>0$. If we define on $\Sg$ the unit normal vector $N(p):=\eta_K\big((p-p_0)/\la\big)$ then, by using \eqref{eq:normal} and equality $(\pi_K\circ\eta_K)(w)=w$ for any $w\in\ptl K$, we get $N_K(p)=(p-p_0)/\la$ for any $p\in\Sg$.

The fact that the base set $\mathcal{D}\subset\sph^n$ of the cone $\cc$ is a strictly convex domain means that the second fundamental form of $\ptl\mathcal{D}$ as a hypersurface of $\sph^n$ is always definite positive with respect to the inner unit normal. This implies that $\text{II}_p(w,w)>0$ for any $p\in\ptl\Sg$ and any  $w\in T_p(\ptl\cc)\setminus\{0\}$ orthogonal to $p$. From the identity $\text{II}(N_K,N_K)=0$ at the end of the proof of Theorem~\ref{th:main}, it follows that $(p-p_0)/\la$ is proportional to $p$, for any $p\in\ptl\Sg$. From here we get $p_0=0$, and this completes the proof.
\end{proof}

A special example of convex cone is a half-space $\mathcal{H}\subset\rrn$. Since $\ptl\mathcal{H}$ is smooth we do not need to assume $\ptl\Sg\subset\ptl\mathcal{H}\setminus\{0\}$. In this case we can deduce that $\ptl K\cap\overline{\mathcal{H}}$ is the unique compact anisotropic stable hypersurface in $\mathcal{H}$, up to translations along $\ptl\mathcal{H}$. A similar result for anisotropic stable capillary hypersurfaces in $\mathcal{H}$ has been given by Guo and Xia~\cite[Thm.~1.1]{guo-xia} with a different proof.

\begin{corollary}
\label{cor:halfspace}
Let $\Sg$ be a compact, connected, two-sided hypersurface immersed in an open half-space $\mathcal{H}\subset\rrn$ with smooth boundary $\ptl\Sg\subset\ptl\mathcal{H}$ such that $\Sg\cap\ptl\mathcal{H}=\ptl\Sg$. If $\Sg$ is anisotropic stable, then $\Sg=p_0+\la\,(\ptl K\cap\overline{\mathcal{H}})$ for some $p_0\in\ptl\mathcal{H}$ and $\la>0$.
\end{corollary}

\begin{proof}
After applying Theorem~\ref{th:main} it remains to see that $p_0\in\ptl\mathcal{H}$. By reasoning as in the proof of Corollary~\ref{cor:stronger} we obtain $N_K(p)=(p-p_0)/\la$, for any $p\in\Sg$. The orthogonality condition $\escpr{N_K,\xi}=0$ along $\ptl\Sg$ entails that $p-p_0\in T_p(\ptl\mathcal{H})$ for any $p\in\ptl\Sg$. As $\ptl\mathcal{H}$ is a Euclidean hyperplane, the straight line starting from a point $p\in\ptl\Sg$ with $p\neq p_0$ and generated by the vector $p-p_0$ is entirely contained in $\ptl\mathcal{H}$. This shows that $p_0\in\ptl\mathcal{H}$, as we claimed.
\end{proof}

\begin{remark}[Non-smooth cones]
The proof of Theorem~\ref{th:main} is still valid when $\cc$ is an arbitrary open convex cone (the base domain $\mathcal{D}\subset\sph^n$ need not be smooth), provided the boundary $\ptl\Sg$ is contained in a smooth open portion of $\ptl\cc$. For instance, the conclusion holds for a compact anisotropic stable hypersurface $\Sg$ disjoint from the edge of a domain $\cc$ bounded by two transversal hyperplanes. In this direction, Koiso~\cite[Thm.~4]{koiso-equilibrium}, \cite[Thm.~1]{koiso-2023} has characterized compact anisotropic stable capillary hypersurfaces in wedge-shaped domains of $\rrn$.
\end{remark}

\begin{remark}[Planar cones]
For a solid cone $\cc\subset\rr^2$, the boundary $\ptl\cc$ is the union of two-closed half-lines leaving from $0$. Thus, we have $\text{II}=0$ along $\ptl\cc\setminus\{0\}$ and the boundary term in \eqref{eq:desi} disappears. Hence, the conclusion of Theorem~\ref{th:main} remains valid even if the cone is not convex.
\end{remark}

\begin{remark}[Minkowski formula in $\cc$]
For a compact anisotropic stationary hypersurface $\Sg$ in a solid cone $\cc$, we can use \eqref{eq:volume} to write \eqref{eq:minkowski} as
\[
\int_\Sg\big(\var_K+H_K\escpr{p,N(p)}\big)\,d\Sg=0.
\]
Our proof of \eqref{eq:minkowski} implies that the previous identity is true for any compact hypersurface $\Sg$ immersed in $\cc$ with boundary $\ptl\Sg$ in $\ptl\cc\setminus\{0\}$ and satisfying $\escpr{N_K,\xi}=0$ along $\ptl\Sg$. Similar formulas were previously obtained by He and Li~\cite[Thm.~1.1]{he-li}, and Jia, Wang, Xia and Zhang~\cite[Thm.~1.3]{jwxz}. 
\end{remark}

\providecommand{\bysame}{\leavevmode\hbox to3em{\hrulefill}\thinspace}
\providecommand{\MR}{\relax\ifhmode\unskip\space\fi MR }
\providecommand{\MRhref}[2]{
  \href{http://www.ams.org/mathscinet-getitem?mr=#1}{#2}
}
\providecommand{\href}[2]{#2}

\end{document}